\documentclass[11pt]{amsart}
\usepackage{mathrsfs,latexsym,amsfonts,amssymb}
\usepackage{hyperref}
\newtheorem{theorem}{Theorem}[section]
\newtheorem{proposition}[theorem]{Proposition}
\newtheorem{lemma}[theorem]{Lemma}
\newtheorem{corollary}[theorem]{Corollary}
\newtheorem{question}[theorem]{Question}

\newtheorem{example}[theorem]{Example}
\theoremstyle{definition}
\newtheorem{definition}[theorem]{Definition}
\theoremstyle{remark}
\newtheorem{remark}[theorem]{Remark}

\begin{document}
\title[Some properties of Pre-uniform spaces]
{Some properties of Pre-uniform spaces}

\author{Fucai Lin*}
\address{Fucai Lin: 1. School of mathematics and statistics,
Minnan Normal University, Zhangzhou 363000, P. R. China; 2. Fujian Key Laboratory of Granular Computing and Application, Minnan Normal University, Zhangzhou 363000, China}
\email{linfucai@mnnu.edu.cn; linfucai2008@aliyun.com}

\author{Yufan Xie}
\address{Yufan Xie: 1. School of mathematics and statistics,
Minnan Normal University, Zhangzhou 363000, P. R. China}
\email{981553469@qq.com}

\author{Ting Wu}
\address{Ting Wu: 1. School of mathematics and statistics,
Minnan Normal University, Zhangzhou 363000, P. R. China}
\email{473207095@qq.com}

\author{Meng Bao}
\address{Meng Bao: College of Mathematics, Sichuan University, Chengdu 610064, China}
\email{mengbao95213@163.com}

\thanks{The first author is supported by the Key Program of the Natural Science Foundation of Fujian Province (No: 2020J02043) and the NSFC (No. 11571158).}
\thanks{* Corresponding author}


\date{\today}
\begin{abstract}
In this paper, we introduce the notions of pre-uniform spaces and pre-proximities and investigate some basic properties about them, where the definition of pre-uniformity here is different with the pre-uniformities which are studied in \cite{BR2016}, \cite{GM2007} and \cite{K2016} respectively. First, we prove that each pre-uniform pre-topology is regular, and give an example to show that there exists a pre-uniform structure on a finite set such that the pre-uniform pre-topology is not discrete. Moreover, we give three methods of generating (strongly) pre-uniformities, that is, the definition of a pre-base, a family of strongly pre-uniform covers, or a family of strongly pre-uniform pseudometrics. As an application, we show that each strongly pre-topological group is completely regular. Finally, we pose the concept of the pre-proximity on a set and discuss some properties of the pre-proximity.
\end{abstract}

\maketitle
\section{Introduction}
The concepts of a uniform space and of a proximity space can be considered either as axiomatizations of some geometric notions, close to but quite independent of the concept of a topological space, or as convenient tools for an investigation of topological spaces. As we all know, uniformities and proximities both can be applied as topological tools. Indeed, the theory of uniform spaces shows striking analogies with the theory of metric spaces, but the realm of its applicability is much broader, see \cite{Eng, I1964, S1952, T1940, W1938, W2019}. In particular, uniformity is a useful tool when we research the theory of topological groups.

In 1999, Doignon and Falmagne 1999 introduced the theory of knowledge spaces (KST) which is regarded as a mathematical framework for the assessment of knowledge and advices for further learning \cite{doignon1985spaces,falmagne2011learning}. KST makes a dynamic evaluation process; of course, the accurate dynamic evaluation is based on individuals' responses to items and the quasi-order on domain $Q$\cite{doignon1985spaces}. In 2009, Danilov discussed the knowledge spaces based on the topological point of view. Indeed, the notion of a knowledge space is a generalization of topological spaces \cite{2009Danilov}, that is, a {\it generalized topology} on a set $Z$ is a subfamily $\mathscr{T}$ of $2^{Z}$ such that $\mathscr{T}$ is closed under arbitrary unions.
Cs\'{a}sz\'{a}r (2002) in \cite{Generalized2002} introduced the notions of generalized topological spaces and then investigated some properties of generalized topological spaces, see \cite{Generalized2002,Generalized2004, Generalized2008,Generalized2009}. Further, J. Li first discuss the pre-topology (that is, the subbase for the topology) with the applications in the theory of rough sets, see \cite{lijinjin2004,lijinjin2006,lijinjin2007}, and then D. Liu in \cite{liudejin2011,liudejin2013} discuss some properties of pre-topology. Recently, Lin, Cao and Li \cite{LCL} have systematically investigated some properties of pre-topology, and Lin, Wu, Xie and Bao in \cite{LWXB} have introduced the concept of pre-topological group and studied some properties of it.
Since the concept of uniformity plays an important role in the study of topological spaces, it is natural to pose the concept of pre-uniform structure as pre-topological tools in the applications of pre-topological spaces, where the definition of pre-uniformity here is different with the pre-uniformities which are studied in \cite{BR2016}, \cite{GM2007} and \cite{K2016} respectively..

This paper is organized as follows. In Section 2, we introduce the
necessary notation and terminology which are used in
the paper. In Section 3, we introduce the notion of pre-uniformity structure and investigate some properties about pre-uniformities. We find that some properties of the pre-uniform structures which are similar to that of uniform spaces hold and others do not hold. For example, compared with the property that each uniform structure on a finite set is discrete uniformity, we give an example that a pre-uniform space $\mu$ on a finite set $X$ such that pre-uniform pre-topology $(X, \tau(\mu))$ is not discrete. In Section 4, we introduce three methods of generating (strongly) pre-uniformities, that is, the definition of a pre-base, a family of strongly pre-uniform covers, or a family of strongly pre-uniform pseudometrics. As an application, we prove that each strongly pre-topological group is completely regular. In Section 5, we discuss the relation of pre-uniformly continuous and continuous in pre-uniform spaces. In Section 6, we introduce pre-proximities and pre-proximity spaces, and then we investigate some properties about them.

\smallskip
\section{Introduction and preliminaries}
Denote the sets of real number, positive integers, the closed unit interval and all non-negative integers by $\mathbb{R}$,  $\mathbb{N}$, $I$ and $\omega$, respectively. Readers may refer
\cite{Eng, LCL} for terminology and notations not
explicitly given here.

We recall some concepts about pre-topological spaces.

\begin{definition}\cite{Generalized2002,lijinjin2004, 2009Danilov}
A {\it pre-topology} on a set $Z$ is a subfamily $\mathscr{T}$ of $2^{Z}$ such that $\bigcup\mathscr{T}=Z$ and $\bigcup\mathscr{T}^{\prime}\in\mathscr{T}$  for any $\mathscr{T}^{\prime}\subseteq\mathscr{T}$. Each element of $\mathscr{T}$ is called an {\it open set} of the pre-topology.
\end{definition}

\begin{definition}\cite{LCL}
A subset $D$ of a pre-topological space $Z$ is {\it closed} provided $Z\setminus D$ is open in $Z$.
\end{definition}

Take an arbitrary subset $F$ of a pre-topological space $(Z, \tau)$;  then it follows that $$\bigcap\{C: F\subseteq C, Z\setminus C\tau\}$$ is closed in $Z$, which is called the {\it closure} \cite{LCL} of $F$ and denoted by $\overline{F}$. Clearly, $\overline{F}$ is the smallest closed set containing $F$, and a set $C$ is closed iff $C=\overline{C}$.

\begin{definition}\cite{LCL}
If $B$ is a subset of a pre-topological space $(Z, \tau)$, then the set $$\bigcup\{W\subseteq B: W\in\tau\}$$ is called the {\it interior} of $B$ and is denoted by $B^{\circ}$.
\end{definition}

\begin{definition}\cite{LCL}
Let $(G, \tau)$ be a pre-topological space and $\mathcal{B}\subseteq \tau$. If for each $U\in\tau$ there exists a subfamily  $\mathcal{B}^{\prime}$ of $\mathcal{B}$ such that $U=\bigcup\mathcal{B}^{\prime}$, then we say that $\mathcal{B}$ is a {\it pre-base} of $(G, \tau)$.
\end{definition}

\begin{definition}\cite{LCL}
Let $h: Y\rightarrow Z$ be a mapping between two pre-topological spaces $(Y, \tau)$ and $(Z, \upsilon)$. The mapping $h$ is {\it pre-continuous} from $Y$ to $Z$ if $h^{-1}(W)\in \tau$ for each $W\in\upsilon$.
\end{definition}

\begin{definition}\cite{LCL}
Let $(Z, \tau)$ be a pre-topological space. Then

\smallskip
(1) $Z$ is called a {\it $T_{0}$-space} if for any $y, z\in Z$ with $y\neq z$ there exists $W\in\tau$ such that $W\cap \{y, z\}$ is exact one-point set;

\smallskip
(2) $Z$ is called a {\it $T_{1}$-space} if for any $y, z\in Z$ with $y\neq z$ there are $V, W\in\tau$ so that $V\cap\{y, z\}=\{y\}$ and $W\cap\{y, z\}=\{z\}$;

\smallskip
(3) $Z$ is called a {\it $T_{2}$-space}, or a {\it Hausdorff space}, if for any $y, z\in Z$ with $y\neq z$ there are $V, W\in\tau$ so that $y\in V$, $z\in W$ and $V\cap W=\emptyset$;

\smallskip
(4) $Z$ is called a {\it $T_{3}$ pre-topological space}, or a {\it regular space}, if $Z$ is a $T_{1}$-space, and for every $z\in Z$ and every closed set $A$ of $Z$ with $z\not\in A$ there are open subsets $V$ and $O$ such that $V\cap O=\emptyset$, $z\in V$ and $A\subseteq O$;

\smallskip
(5) $Z$ is  a {\it $T_{3\frac{1}{2}}$ pre-topological space}, or a {\it completely regular pre-topological space}, or a {\it Tychonoff pre-topological space}, if $Z$ is a $T_{1}$-space, and for each $z\in Z$ and each closed subset $C\subseteq Z$ with $z\not\in C$ there exists a pre-continuous mapping $r: Z\rightarrow I$ so that $r(z)=0$ and $r(x)=1$ for each $x\in C$.
\end{definition}

\section{Basic properties of pre-uniformities}
In this section, we mainly introduce the concept of pre-uniformity and discuss some basic properties of it.

Let $X$ be a nonempty set. We say that the set $\bigtriangleup=\{(x, x): x\in X\}$ is the {\it diagonal} of $X\times X$. For any subsets $A, B$ of $X$ and $x\in X$,  denote by $$A^{-1}=\{(y, x): (x, y)\in A\},$$$$A\circ B=\{(x, y): \mbox{there exists}\ z\in X\ \mbox{such that}\ z\in A\ \mbox{and}\ (z, y)\in B\},$$
$$A[x]=\{y\in X: (x, y)\in A\}.$$If $A=A^{-1}$, we say that $A$ is {\it symmetric}.

\begin{definition}
Let $\mu$ be a family of non-empty subsets of $X\times X$ such that the following conditions are satisfied

\smallskip
(U1) for any $U\in\mu$, $\triangle\subseteq U$;

\smallskip
(U2) if $U\in\mu$, then $U^{-1}\in\mu$;

\smallskip
(U3) if $U\in\mu$, then there exist $V, W\in\mu$ such that $V\circ W\subseteq U$;

\smallskip
(U4) if $U\in\mu$ and $U\subset V\subseteq X\times X$, then $V\in\mu$;

\smallskip
(U5) if $\bigcap\mu=\bigtriangleup$;

\smallskip
(U6) if $U, V\in\mu$, then $U\cap V\in\mu$;

\smallskip
(U2$^{\prime}$) if $U\in\mu$, there exists $V\in\mu$ such that $V\subset U$ and $V=V^{-1}$;

\smallskip
(U3$^{\prime}$) if $U\in\mu$, then there exists $V\in\mu$ such that $V\circ V\subseteq U$.

\smallskip
$\bullet$ The family $\mu$ is a {\it pre-uniform structure} of $X$ if $\mu$ satisfies (U1)-(U5), the pair $(X, \mu)$ is a {\it pre-uniform space} and the members of $\mu$ are called {\it entourage}.

$\bullet$ If a pre-uniform structure $\mu$ also satisfies (U2$^{\prime}$), then we say that $\mu$ is a {\it symmetrically pre-uniform structure} and $(X, \mu)$ is a {\it symmetrically pre-uniform space}.

$\bullet$ If a pre-uniform structure $\mu$ also satisfies (U3$^{\prime}$), then we say that $\mu$ is a {\it strongly pre-uniform structure} and $(X, \mu)$ is a {\it strongly pre-uniform space}.

$\bullet$ If $\mu$ is a symmetrically and strongly pre-uniform structure $\mu$, then we say that $\mu$ is an {\it almost uniform structure} and $(X, \mu)$ is an {\it almost uniform space}.

$\bullet$ An almost uniform structure $\mu$ satisfies (U6) is called {\it uniform} and $(X, \mu)$ is a {\it uniform space}.
\end{definition}

A family $\mathcal{B}\subset\mu$ is called a {\it pre-base} for the pre-uniformity $\mu$ if for each $V\in\mu$ there exists a $W\in\mathcal{B}$ such that $W\subset V$. The uniformity that has $\{\bigtriangleup\}$ as a pre-base is called the {\it discrete pre-uniformity}. The smallest cardinal number of the form $|\mathcal{B}|$, where $\mathcal{B}$ is a pre-base for $\mu$, is called the {\it weight of the pre-uniformity} $\mu$ and is denoted by $w(\mu)$. Clearly, we have the following proposition.

\begin{proposition}\label{p111}
Let $(X, \mu)$ be a pre-uniform space. If $\mathcal{B}$ is a pre-base for $\mu$, then $\mathcal{B}$ has the following properties:

\smallskip
(BU1) For each $V\in\mathcal{B}$ there exist $U\in\mathcal{B}$ such that $U^{-1}\subset V$.

\smallskip
(BU2) For each $V\in\mathcal{B}$ there exist $U, W\in\mathcal{B}$ such that $U\circ W\subset V$.

\smallskip
(BU3) $\bigcap\mathcal{B}=\triangle$.
\end{proposition}

\begin{remark}
Clearly, each uniform space is an almost uniform space, each almost uniform space is strongly pre-uniform space and symmetrically pre-uniform space, and each strongly pre-uniform space or symmetrically pre-uniform space is a pre-uniform space, but not vice versa, see the following examples.

\smallskip
(1) There exists a strongly pre-uniform space which is not a symmetrically pre-uniform space; in particular, it is not an almost uniform space. Indeed, let $X=\mathbb{R}$, and let $\mathscr{A}=\{(x, +\infty): x\in \mathbb{R}\}\cup \{(-\infty, x): x\in \mathbb{R}\}$. Let $\mathcal{B}$ be the set of all the forms $\bigcup_{x\in\mathbb{R}}\left((\{x\}\times A_{x})\cup(\{B_{x}\times \{x\})\right)$, where $A_{x}, B_{x}\in\mathscr{A}$ and $x\in A_{x}\cap B_{x}$ for each $x\in \mathbb{R}$, and let $\mu$ be the pre-uniform generated by $\mathcal{B}$. Then it easily check that $(\mathbb{R}, \mu)$ is a strongly pre-uniform space which is not a symmetrically pre-uniform space since the element $\bigcup_{x\in\mathbb{R}}\left((\{x\}\times A_{x})\cup(\{B_{x}\times \{x\})\right)$ does not satisfy (U2$^{\prime}$), where $A_{x}\in\{(t, +\infty): t\in \mathbb{R}\}$ and $B_{x}\in\{(-\infty, t): t\in \mathbb{R}\}$ for each $x\in \mathbb{R}$.

\smallskip
(2) There exists an almost uniform space which is not a uniform space. Indeed, let $X$ be any non-empty set such that there exist two uniform structures $\mu_{1}$ and $\mu_{2}$ on the set $X$ satisfying that there are $U_{0}\in\mu_{1}$ and $V_{0}\in\mu_{2}$ such that for any $U\in\mu_{1}$ and $V\in\mu_{2}$ we have $U\nsubseteq V_{0}$ and $V\nsubseteq U_{0}$, see \cite[8.1.B]{Eng}. Put $\delta=\mu_{1}\cup\mu_{2}$, and let $\mu$ be the pre-uniform structure generated by $\delta$. Then $(X, \mu)$ is an almost uniform space which is not a uniform space since $U_{0}\cap V_{0}\not\in\mu$.
\end{remark}

However, the following question is still unknown for us.

\begin{question}
Does there exist a symmetrically pre-uniform space which is not a strongly pre-uniform space?
\end{question}

The proof of the following proposition is easy, thus we omit it.

\begin{proposition}
Let $(X, \mu)$ be a  pre-uniform space. Put $$\tau(\mu)=\{G\subseteq X: \mbox{for each}\ x\in G\ \mbox{there exists}\ U\in\mu\ \mbox{such that}\ U[x]\subseteq G\}.$$Then $\tau(\mu)$ is a pre-topology on $X$.
\end{proposition}

We say that $\tau(\mu)$ is the {\it induced pre-topology} from the pre-uniform structure $(X, \mu)$, or say that $\tau(\mu)$ is the {\it pre-uniform pre-topology} of $(X, \mu)$. If $X$ is a pre-topological space and a pre-uniformity $\mu$ on the set $X$ induces the original pre-topology of $X$, then we say that $\mu$ is a {\it pre-uniformity} on the pre-topological space $X$.

If $\mu_{1}$ and $\mu_{2}$ are two pre-uniformities on a set $X$ and $\mu_{2}\subset\mu_{1}$, then we say that the pre-uniformity $\mu_{1}$ is {\it finer} than the pre-uniformity $\mu_{2}$ or that $\mu_{2}$ is {\it coarser} than $\mu_{1}$. It is easily checked the following two propositions hold.

\begin{proposition}
If a pre-uniformity $\mu_{1}$ on a set is finer than a pre-uniformity $\mu_{2}$, then $\tau(\mu_{1})$ is finer than $\tau(\mu_{2})$.
\end{proposition}

\begin{proposition}
If $\{\mu_{s}\}_{s\in S}$ is a family of pre-uniformities on a set $X$, then there exists a pre-uniformity $\mu$ on $X$ which is coarser than any pre-uniformity on $X$ that is finer than all pre-uniformities $\mu_{s}$. Moreover, if the pre-uniformity $\mu_{s}$ induces the pre-topology $\tau(\mu_{s})$ for each $s\in S$, then the pre-topology induced by the least upper bound of the family $\{\mu_{s}\}_{s\in S}$ is the least upper bound of the family $\{\tau(\mu_{s})\}_{s\in S}$ of pre-topologies on the set $X$.
\end{proposition}

\begin{lemma}\label{l7}
Let $(X, \mu)$ be a pre-uniform space. For each $x\in X$, put $\mu_{x}=\{(U[x])^{\circ}: U\in\mu\}$, where each $(U[x])^{\circ}$ denotes the interior of $U[x]$ in the induced pre-topology $\tau(\mu)$. Then $\mu_{x}$ is an open neighborhood pre-base at $x$.
\end{lemma}

\begin{proof}
Take any $x\in X$ and $U\in\mu$. Put $$G=\{y\in X: \mbox{there exists}\ V\in\mu\ \mbox{such that}\ V[y]\subseteq U[x]\}.$$Then $x\in G\subset U[x]$. Hence it suffices to prove that $G$ is open in the pre-uniform pre-topology. For each $y\in G$, there exists $V\in\mu$ such that $V[y]\subseteq U[x]$, hence there exist $W_{1}, W_{2}$ such that $W_{1}\circ W_{2}\subseteq V$. For any $u\in W_{1}[y]$ and $v\in W_{2}[u]$, we have $(y, u)\in W_{1}$ and $(u, v)\in W_{2}$, hence $(y, v)\in W_{1}\circ W_{2}\subseteq V$, then $v\in V[y]\subseteq U[x]$. Therefore, $W_{2}[u]\subseteq U[x]$, then $u\in G$. By the arbitrary choice of $u\in W_{1}[y]$, we have $W_{1}[y]\subseteq G$. Hence $G$ is open in $X$.
\end{proof}

\begin{lemma}\label{l6}
Let $(X, \mu)$ be a pre-uniform space. Put $$\beta=\{B\in\mu: B\ \mbox{is closed in}\ X\times X\}\ \mbox{and}\ \lambda=\{C\in\mu: C\ \mbox{is open in}\ X\times X\}.$$
Then both $\beta$ and $\lambda$ are pre-bases for $\mu$.
\end{lemma}

\begin{proof}
We first prove that $\beta$ is a pre-base for $\mu$. Take any $U\in \mu$. Then there exist $V_{1}, V_{2}, V_{3}\in\mu$ and $W\in\mu$ such that $V_{1}\circ V_{2}\circ V_{3}\subseteq U$ and $W^{-1}\subseteq V_{3}$. Let $(x, y)\in \overline{V_{2}}$. Since $(V_{1}[x]\times W[y])\cap V_{2}\neq\emptyset$, there exists $(s, t)\in (V_{1}[x]\times W[y])\cap V_{2}$, hence $(x, s)\in V_{1}, (s, t)\in V_{2}$ and $(y, t)\in W$, then $(x, s)\in V_{1}, (s, t)\in V_{2}$ and $(t, y)\in W^{-1}$, which implies that $(x, y)\in V_{1}\circ V_{2}\circ V_{3}$. Therefore, $\overline{V_{2}}\subseteq V_{1}\circ V_{2}\circ V_{3}\subseteq U$.

Now we prove that $\lambda$ is a pre-base for $\mu$. Take any $U\in \mu$. Then there exist $W_{1}, W_{2}, W_{3}\in\mu$ and $O_{1}\in\mu$ such that $W_{1}\circ W_{2}\circ W_{3}\subseteq U$ and $O_{1}^{-1}\subseteq W_{1}$. Let $(x, y)\in W_{2}$. Then $W_{1}^{-1}[x]\times W_{3}[y]$ is a neighborhood of $(x, y)$. We claim that $W_{1}^{-1}[x]\times W_{3}[y]\subset U$. Indeed, take any $(u, v)\in  W_{1}^{-1}[x]\times W_{3}[y]$. Then $(u, x)\in W_{1}$, $(x, y)\in W_{2}$ and $(y, v)\in W_{3}$, hence $(u, v)\in W_{1}\circ W_{2}\circ W_{3}\subseteq U$. Therefore, $W_{2}\subset U^{\circ}$. Hence $\lambda$ is a pre-base for $\mu$.
\end{proof}

\begin{remark}
If $V$ is closed in $X\times X$ in the pre-uniform structure $(X, \mu)$, then $V[x]$ is closed in $X$ since for any fixed $x\in X$ the mapping $X\rightarrow X\times X$ defined by $y\mapsto (x, y)$ for any $y\in X$ is pre-continuous.
\end{remark}

\begin{lemma}\label{l5}
Let $(X, \mu)$ be a pre-uniform space. Then the pre-uniform pre-topology $(X, \tau(\mu))$ is $T_{0}$ if and only if $\triangle=\bigcap\mu$.
\end{lemma}

\begin{proof}
Necessity. Assume $(X, \tau(\mu))$ is $T_{0}$. Then for any $(x, y)\in X\times X\setminus \triangle$, there exists $U\in\mu$ such that $y\not\in U[x]$ or $x\not\in U[y]$. If $y\not\in U[x]$, then it is obvious that $(x, y)\not\in U$. If $x\not\in U[y]$, then there exists $V\in\mu$ such that $V^{-1}\subseteq U$, then $x\not\in V^{-1}[y]$, thus $(y, x)\not\in V^{-1}$, that is,  $(x, y)\not\in V$. Therefore, we have $\triangle=\bigcap\mu$.

Sufficiency. Assume that $\triangle=\bigcap\mu$. Take any distinct points $x$ and $y$. Then there exists $U\in\mu$ such that $(x, y)\not\in U$, hence $y\not\in U[x]$. Therefore, $(X, \tau(\mu))$ is $T_{0}$.
\end{proof}

\begin{proposition}\label{p8}
If $(X, \mu)$ is a pre-uniform space, then the pre-uniform pre-topology $(X, \tau(\mu))$ is ($T_{1}$) regular.
\end{proposition}

\begin{proof}
By Lemma~\ref{l5}, $(X, \tau(\mu))$ is $T_{0}$. First, we prove that it is Hausdorff. Indeed, take any distinct points $x$ and $y$. By Lemma~\ref{l5}, there exists $U\in\mu$ such that $(x, y)\not\in U$. We can find $V, W\in\mu$ such that $V\circ W\subseteq U$. We claim that $V[x]\cap W^{-1}[y]=\emptyset$. Suppose not, take any $z\in V[x]\cap W^{-1}[y]$. Then $(x, z)\in V$ and $(y, z)\in W^{-1}$, hence $(x, y)\in V\circ W\subseteq U$, which is a contradiction.

Now we prove that $X$ is regular. Take any $x\in X$ and $U\in\mu$. Then there exists $V, W\in\mu$ such that $V\circ W\subset U$. We claim that $\overline{V[x]}\subset U[x]$. Indeed, pick any $y\in \overline{V[x]}$; then $W^{-1}[y]\cap V[x]\neq\emptyset$, hence take any $z\in W^{-1}[y]\cap V[x]$. Therefore, $(y, z)\in W^{-1}$ and $(x, z)\in V$, then it follows that $(x, y)\in V\circ W\subset U$, thus $y\in U[x]$.
\end{proof}

Each uniform structure on a finite set is discrete uniformity, but the situation is different in the class of pre-uniform structure, see the following example.

\begin{example}
There exists a pre-uniform space $\mu$ on a finite set $X$ such that pre-uniform pre-topology $(X, \tau(\mu))$ is not discrete.
\end{example}

\begin{proof}
Let $X=\{a, b, c\}$, and let $$U_{1}=\{(a, a), (b, b), (c, c), (a, b), (b, c), (c, a)\}$$ and
$$U_{2}=\{(a, a), (b, b), (c, c), (a, c), (b, a), (c, b)\}.$$ Then the family $\mathcal{B}=\{U_{1}, U_{2}\}$ satisfies  the conditions (BU1)-(BU3) in Proposition~\ref{p111}. Let $\mu$ be the pre-uniform generated by the family $\mathcal{B}$. Then it is easily checked that the pre-uniform pre-topology $(X, \tau(\mu))$ is not discrete since $\triangle\not\in\mu$. Moreover, it is obvious that $\mu$ is not an almost uniformity.
\end{proof}

Consider a pre-uniform space (resp. an almost uniform space) $(X, \mu)$ and a pseudometric $\rho$ on the set $X$; we say that the pseudometric $\rho$ is pre-uniform (resp. almost uniform) with respect to $\mu$ if for each $\varepsilon>0$ there is a $V\in\mu$ such that $\rho(x, y)<\varepsilon$ whenever $(x, y)\in V$.

\begin{proposition}\label{p89}
If a pseudometric $\rho$ on a set $X$ is pre-uniform with respect to a pre-uniformity $\mu$ on $X$, then $\rho$ is a pre-continuous function from the set $X\times X$ with the pre-topology induced by the pre-uniformity $\mu$ to the real line.
\end{proposition}

\begin{proof}
Let $(x_{0}, y_{0})$ be a point of $X\times X$; pick an $\varepsilon>0$ and a $V\in\mu$ such that $\rho(x, y)<\frac{\varepsilon}{2}$ for any $(x, y)\in V$. From Lemma~\ref{l7}, the set $(V[x_{0}])^{\circ}\times (V[y_{0}])^{\circ}$ is an open neighborhood of $(x_{0}, y_{0})$, hence it only need to prove that $$|\rho(x_{0}, y_{0})-\rho(x, y)|<\varepsilon\ \mbox{for each}\ (x, y)\in V[x_{0}]\times V[y_{0}].$$ However,  if $(x, y)\in V[x_{0}]\times V[y_{0}]$, then $(x_{0}, x)\in V$ and $(y_{0}, y)\in V$, hence it follows from the triangle inequality that $$|\rho(x_{0}, y_{0})-\rho(x, y)|<\rho(x_{0}, x)+\rho(y_{0}, y)<\frac{\varepsilon}{2}+\frac{\varepsilon}{2}=\varepsilon.$$
\end{proof}

Since all the open sets of a pre-topology is a subbase of a topological space, it follows from \cite[Theorem 8.1.10]{Eng} that we have the following theorem.

\begin{theorem}\label{t6}
For each sequence $V_{0}, V_{1}, \ldots, $of members of a pre-uniformity $\mu$ on a set $X$, where $$V_{0}=X\times X, V_{i+1}\circ V_{i+1}\circ V_{i+1}\subset V_{i}\ \mbox{and}\ V_{i}^{-1}=V_{i}\ \mbox{for any}\ i\in\mathbb{N},$$there exists a pseudometric $\rho$ on the set $X$ such that for each $i\geq 1$$$\{(x, y): \rho(x, y)<\frac{1}{2^{i}}\}\subset V_{i}\subset \{(x, y): \rho(x, y)\leq\frac{1}{2^{i}}\}$$
\end{theorem}

\begin{corollary}\label{c8}
For each strongly pre-uniformity $\mu$ on a set $X$ and any $V\in\mu$ there exists a pseudometric $\rho$ on the set $X$ with respect to $\mu$ and satisfies the following condition $$\{(x, y): \rho(x, y)<1\}\subset V.$$
\end{corollary}

\begin{proof}
By the definition of strongly pre-uniformity, there exists a sequence $V_{0}, V_{1}, \ldots, $of members of $\mu$ such that $$V_{0}=X\times X, V_{1}=V_{1}^{-1}\subset V\ \mbox{and}\ V_{i+1}\circ V_{i+1}\circ V_{i+1}\subset V_{i}\ \mbox{for any}\ i\in\mathbb{N}.$$ Put $\rho=2\rho_{0}$, where $\rho_{0}$ is a pseudometric satisfying Theorem~\ref{t6}, has the required property.
\end{proof}

Let $(X, \mu)$ be a strongly pre-uniform space, and let $\mathcal{P}$ be the family of all pseudometrics on the set $X$ that are strongly pre-uniform with respect to $\mu$. By Corollary~\ref{c8}, we have the following proposition.

\begin{proposition}\label{p9}
For each pair $x, y$ of distinct points of $X$, there exists a $\rho\in\mathcal{P}$ such that $\rho(x, y)>0$.
\end{proposition}

Now we can prove our main result in this section.

\begin{theorem}\label{t7}
For each strongly pre-uniformity $\mu$ on a set $X$, the pre-topology $(X, \tau(\mu))$ is completely regular.
\end{theorem}

\begin{proof}
Take any $x\in X$ and any closed set $F\subset X$ with $x\not\in F$. By Proposition~\ref{p8} and Lemma~\ref{l7}, there exists $V\in\mu$ such that $V[x]\cap F=\emptyset$. It follows from Corollary~\ref{c8} that there exists a pseudometric $\rho$ on the set $X$ with respect to $\mu$ and satisfies the following condition $$\{(x, y): \rho(x, y)<1\}\subset V.$$ By Proposition~\ref{p89}, $\rho$ is pre-continuous. Define the function $f: X\rightarrow I$ by $f(y)=\min\{1, \rho(x, y)\}$ for each $y\in X$. It easily see that $f$ is pre-continuous, vanishes at $x$ and is equal to one on $F$.
\end{proof}

The following question is still unknown for us.

\begin{question}
For each pre-uniformity $\mu$ on a set $X$, is the pre-topology $(X, \tau(\mu))$ completely regular?
\end{question}

\smallskip
\section{The complete regularity of pre-uniform spaces}
In this section, we introduce three methods of generating (strongly) pre-uniformities, that is, the definition of a pre-base, a family of strongly pre-uniform covers, or a family of strongly pre-uniform pseudometrics. As an application, we show that each strongly pre-topological group is completely regular.

For any $V\in \mathcal{D}_{X}$ of the set $X$, put $\mathscr{C}(V)=\{V[x]\}_{x\in X}$, where $\mathcal{D}_{X}=\{U\subset X\times X: \bigtriangleup\subset U\}$; then $\mathscr{C}(V)$ is a cover of $X$. Let $\mu$ be a pre-uniformity on set $X$; any cover of the set $X$ that has a refinement of the form $\mathscr{C}(V)$ for some $V\in\mu$, is called {\it pre-uniform with respect to $\mu$}. Let $\mathscr{C}$ be the set of all covers of a set $X$ that are pre-uniform with respect to a pre-uniformity $\mu$ on the set $X$. Then we have the following proposition.

\begin{proposition}
If $\mu$ is a strongly pre-uniformity on the set $X$, then $\mathscr{C}$ has the following properties:

\smallskip
(UC1) If $\mathcal{A}\in\mathscr{C}$ and $\mathcal{A}$ is a refinement of a cover $\mathcal{B}$ of the set $X$, then $\mathcal{B}\in\mathscr{C}$.

\smallskip
(UC2) For each $\mathcal{A}\in\mathscr{C}$, there exists a $\mathcal{B}\in\mathscr{C}$ which is a star refinement of $\mathcal{A}$.

\smallskip
(UC3) For each pair $x, y$ of distinct points of $X$ there is an $\mathcal{A}\in\mathscr{C}$ so that no member of $\mathcal{A}$ contains both $x$ and $y$.
\end{proposition}

\begin{proof}
(UC1) is obvious. We need to prove (UC2) and (UC3).

\smallskip
(UC2). Clearly, it suffices to prove that for each $\mathcal{A}=\mathscr{C}(V)\in\mathscr{C}$ the cover $\mathcal{B}=\mathscr{C}(W)$, where $V, W\in\mu$ with $W^{3}\subset V$ and $W^{-1}=W$, is a star refinement of $\mathcal{A}$. Indeed, take any $x\in X$; we claim that $\mbox{st}(W[x], \mathcal{B})\subset V[x]\in\mathcal{A}$. In fact, for any $y\in\mbox{st}(W[x], \mathcal{B})$ there exists $z\in X$ such that $y\in W[z]$ and $W[z]\cap W[x]\neq\emptyset$; pick any $h\in W[z]\cap W[x]$, then $(x, h)\in W$ , $(z, h)\in W$ and $(z, y)\in W$. From $W=W^{-1}$ and $W^{3}\subset V$, we have $(x, y)\in V$, that is, $y\in V[x]$.

\smallskip
(UC3). For each pair $x, y$ of distinct points of $X$, there exists $V\in\mu$ such that $(x, y)\not\in V$. From (U3$^{\prime}$), there is $W\in\mu$ such that $W^{2}\subset V$. Put $\mathcal{A}=\mathscr{C}[W]$. Then it is easily checked that $\mathcal{A}$ satisfies the require property.
\end{proof}

It is more convenient not to describe the family $\mu$ of entourages of the diagonal directly when we define a pre-uniformity on a given set. Here we shall introduce three methods of generating (strongly) pre-uniformities (see Propositions~\ref{p10}, ~\ref{p11} and~\ref{p12}), that is, the definition of a pre-base, a family of strongly pre-uniform covers, or a family of strongly pre-uniform pseudometrics.

\begin{proposition}\label{p11}
Let $X$ be a set, and let $\mathscr{O}$ be a collection of covers of $X$ which has properties (UC1)-(UC3). Put $$\mathcal{B}=\{\bigcup\{A\times A: A\in\mathcal{A}\}: \mathcal{A}\in\mathscr{O}\}.$$ Then $\mathcal{B}$ is a pre-base for a strongly pre-uniformity $\mu$ on the set $X$. The collection $\mathscr{O}$ is the collection of all covers of $X$ which are strongly pre-uniform with respect to $\mu$.

If, moreover, $X$ is a pre-topological space and the collection $\mathscr{O}$ consists of open covers of $X$, and if for each $x\in X$ and each open neighborhood $G$ of $x$ there is $\mathcal{A}\in\mathscr{O}$ such that $\mbox{st}(x, \mathcal{A})\subset G$, then $\mu$ is a strongly pre-uniformity on the space $X$.
\end{proposition}

\begin{proof}
For each cover $\mathcal{A}\in \mathscr{O}$, let $$\mu(\mathcal{A})=\bigcup\{A\times A: A\in\mathcal{A}\};$$then put $$\mathscr{U}=\{\mu(\mathcal{A}): \mathcal{A}\in \mathscr{O}\}.$$Clearly, each $\mu(\mathcal{A})=\mu(\mathcal{A})^{-1}$, thus (BU1) holds. By (UC3), (BU3) also holds. Moreover, it easily check that $\mu(\mathcal{B})\circ \mu(\mathcal{B})\subset \mu(\mathcal{A})$ if $\mathcal{B}$ is a star refinement of $\mathcal{A}$. By (UC1), the collection $\mathscr{O}$ is the collection of all covers of $X$ which are strongly pre-uniform with respect to $\mu$. Further, it is readily established that $\mu(\mathcal{A})[x]=\mbox{st}(x, \mathcal{A}).$
\end{proof}

The following two propositions are trivial.

\begin{proposition}\label{p10}
Let $X$ be a set, and let $\mathcal{B}\subset\mathcal{D}_{X}$ have the properties (BU1)-(BU3) in Proposition~\ref{p111}. Put $$\mu=\{U\in\mathcal{D}_{X}: \mbox{there exists}\ B\in\mathcal{B}\ \mbox{such that}\ B\subset U\}.$$ Then $\mu$ is a pre-uniform structure and the family $\mathcal{B}$ is a pre-base for $\mu$.

If, moreover, $X$ is a pre-topological space and the family $\mathcal{B}$ consists of open subsets of the $X\times X$, and if for each $x\in X$ and each open neighborhood $G$ of $x$ there is $V\in\mathcal{B}$ such that $V[x]\subset G$, then $\mu$ is a pre-uniformity on the pre-topological space $X$.

The pre-uniformity $\mu$ is called the {\it pre-uniformity} generated by the pre-base $\mathcal{B}$.
\end{proposition}

\begin{proposition}\label{p12}
Let $X$ be a set, and let a family $\mathcal{P}_{1}$ of pseudometrics on the set $X$ that satisfies Proposition~\ref{p9}. For each $\rho\in\mathcal{P}_{1}$ and $i\in\mathbb{N}$, let $U_{i, \rho}=\{(x, y): \rho(x, y)<\frac{1}{2^{i}}\}$. Then the family $\mathcal{B}=\{U_{i, \rho}: \rho\in\mathcal{P}_{1}, i\in\mathbb{N}\}$ is a pre-base for a strongly pre-uniformity $\mu$ on the set $X$. Each pseudometric $\rho\in\mathcal{P}_{1}$ is a strongly pre-uniform with respect to $\mu$.

If, moreover, $X$ is a pre-topological space and all pseudometrics of the family $\mathcal{P}_{1}$ are pre-continuous functions from $X\times X$ to the real line, and if for each $x\in X$ and each non-empty closed set $A\subset X$ with $x\not\in A$ there exists a $\rho\in\mathcal{P}_{1}$ such that $\inf\{\rho(x, a): a\in A\}>0$, then $\mu$ is a strongly pre-uniformity on the space $X$.

The strongly pre-uniformity $\mu$ is called the {\it strongly pre-uniformity} generated by the family $\mathcal{P}_{1}$ of strongly pre-uniform pseudometrics.
\end{proposition}

Now we can prove one of main results in this section.

\begin{theorem}\label{t2}
The pre-topology of a pre-topological space $X$ can be induced by a strongly pre-uniformity on the set $X$ if and only if $X$ is a completely regular pre-topological space.
\end{theorem}

\begin{proof}
By Theorem~\ref{t7}, the necessity is obvious. Now assume that $X$ is a completely regular pre-topological space. Denote by $C(X)$ the family of all pre-continuous bounded real-valued functions defined on $X$. For each $f\in C(X)$ the formula $$\rho_{f}(x, y)=|f(x)-f(y)|$$ defines a pseudometric on the set $X$. Put $P=\{\rho_{f}: f\in C(X)\}$. Since $X$ is completely regular, the family $P$ satisfies Proposition~\ref{p9}. Let $\mu$ be the strongly pre-uniformity generated by $C(X)$. We shall prove that the pre-topology by $\mu$ coincide with the original pre-topology of $X$. By Proposition~\ref{p12}, it suffices to prove that for each $x\in X$ and each non-empty closed set $A\subset X$ with $x\not\in A$ there exists a $\rho\in P$ such that $\inf\{\rho(x, a): a\in A\}>0$. However, since $X$ is completely regular, there exists a function $f\in C(X)$ such that $f(x)=0$ and $f(A)\subset\{1\}$, then the pseudometric $\rho_{f}\in P$ satisfies that $\inf\{\rho(x, a): a\in A\}=1$. Therefore, $\mu$ is a strongly pre-uniformity on the pre-topological space $X$.
\end{proof}

\begin{proposition}\label{p888}
Let $X$ be a set and $(X, \mu)$ be a pre-uniform space. If there exists a pseudometric $\rho$ on the set $X$ such that the pre-uniformity induced by $\rho$ coincides with $\mu$, then $(X, \mu)$ is a uniform space.
\end{proposition}

\begin{proof}
Indeed, the family $\{\rho\}$ consisting of the single pseudometric $\rho$ generated a uniformity on the set $X$, hence $(X, \mu)$ be a uniform space.
\end{proof}

\begin{definition}
Let $f$ be a mapping from pre-uniform space $(X, \mu)$ to pre-uniform space $(Y, \nu)$. We say that $f$ is {\it pre-uniformly continuous} if for each $F\in\nu$ there exists $M\in\mu$ such that $\phi(M)\subseteq F$, where $\phi: X\times X\rightarrow Y\times Y$ defined by $\phi(x, z)=(f(x), f(y))$ for each $(x, z)\in X\times X$.
\end{definition}

The following lemma and proposition are trivial.

\begin{lemma}\label{l3}
Let $(X, \mu)$ and $(Y, \nu)$ be pre-uniform spaces. If $f: (X, \mu)\rightarrow(Y, \nu)$ is pre-uniformly continuous, then $f$ is pre-continuous.
\end{lemma}

\begin{proposition}
Let $(X, \mu)$ and $(Y, \nu)$ be pre-uniform spaces and $f$ a mapping of $X$ to $Y$. The following conditions are equivalent:
\begin{enumerate}
\smallskip
\item The mapping $f$ is pre-uniformly continuous with respect to $\mu$ and $\nu$.

\smallskip
\item There exist pre-bases $\mathcal{B}$ and $\mathcal{C}$ for $\mu$ and $\nu$ respectively, such that for each $V\in\nu$ there exists $U\in\mathcal{B}$ satisfying $U\subset (f\times f)^{-1}(V)$.

\smallskip
\item For each cover $\mathcal{A}$ of the set $Y$ which is pre-uniform with respect to $\nu$ the cover $\{f^{-1}(A): A\in\mathcal{A}\}$ of the set $X$ is pre-uniform with respect to $\mu$.

\smallskip
\item For each pseudometric $\rho$ on the set $Y$ which is pre-uniform with to $\nu$ the pseudometric $\sigma$ on the set $X$ by the formula $\sigma(x, y)=\rho(f(x), f(y))$ is pre-uniform with respect to $\mu$.
\end{enumerate}
\end{proposition}

The least upper bound of all pre-uniformities on a completely regular pre-topological space $X$, i.e., the finest pre-uniformity on $X$, is called the {universal pre-uniformity} on the pre-topological space $X$. We say that a pre-uniform space $(X, \mu)$ is {\it fine}, if $\mu$ is the universal pre-uniformity on the pre-topological space $X$ with the pre-topology induced by the pre-uniformity $\mu$.

\begin{proposition}
Each pre-continuous mapping of a completely regular pre-topological space $X$ to a completely regular pre-topological space $Y$ is pre-uniformly continuous with respect to the universal pre-uniformity on the pre-topological space $X$ and any pre-uniformity on the pre-topological space $Y$.
\end{proposition}

\begin{proof}
Let $\mu$ be the universal pre-uniformity on $X$ and $\nu$ be any pre-uniformity on the pre-topological space $Y$. For any $V\in\nu$, we claim that $U_{V}=(f\times f)^{-1}(V)\in\mu$. Indeed, for any $x\in X$, $V[f(x)]$ and $V^{-1}[f(x)]$ are neighborhoods of $f(x)$ in $Y$, then $f^{-1}(V[f(x)])$ and $f^{-1}(V^{-1}[f(x)])$ are neighborhoods of $x$ in $X$ since $f$ is pre-continuous; however, $f^{-1}(V[f(x)])\subset U_{V}[x]=(f\times f)^{-1}(V)[x]$ and $f^{-1}(V^{-1}[f(x)])\subset U_{V}^{-1}[x]=(f\times f)^{-1}(V^{-1})[x]$. Since $\mu$ is the universal pre-uniformity, it follows that $U_{V}=(f\times f)^{-1}(V)\in\mu$.
\end{proof}

Finally, we prove that the second main result in this section.

\begin{definition}\cite{LWXB}
A {\it pre-topological group} $G$ is a group which is also a pre-topological space such that the multiplication mapping of $G\times G$ into $G$ sending $x\times y$ into $x\cdot y$, and the inverse mapping of $H$ into $G$ sending $x$ into $x^{-1}$, are pre-continuous mappings.
\end{definition}

\begin{definition}\cite{LWXB}
If a pre-topological $G$ has a symmetric pre-base $\mathscr{B}_{e}$ at the identity $e$ such that for each $U\in\mathscr{B}_{e}$ there exists $V\in\mathscr{B}_{e}$ so that $V^{2}\subset U$, then we say that $G$ is a {\it strongly pre-topological group}.
\end{definition}

\begin{theorem}\label{t717}
Each strongly pre-topological group $G$ is completely regular.
\end{theorem}

\begin{proof}
Let $\mathcal{B}(e)$ be a pre-base for $G$ at the neutral element $e$. For each $U\in\mathcal{B}(e)$, put $\mathcal{B}(U)=\{xU: x\in G\}$. Let $\mathscr{B}$ be the collection of all covers of $G$ which have a refinement of the form of $\mathcal{B}(U)$, where $U\in\mathcal{B}(e)$. By Theorem~\ref{t7} and Proposition~\ref{p11}, it suffices to prove that $\mathscr{B}$ has properties of (UC1)-(UC3).

Clearly, $\mathscr{B}$ has property of (UC1). For each pair $x, y$ of distinct points of $G$, we have $x^{-1}y\neq e$. By Proposition~\ref{p8}, $G$ is $T_{1}$, hence there exists $U\in\mathcal(B)(e)$ such that $x^{-1}y\in U$. Pick $V\in\mathcal(B)(e)$ such that $V^{-1}V\subset U$. It easily check that no member of the cover $\mathcal{B}(V)=\{xV: x\in G\}$ contains both $x$ and $y$.

To prove that $\mathscr{B}$ has property (UC3) it suffices to show that for each $U\in\mathcal{B}(e)$ there is $V\in\mathcal{B}(e)$ such that $\mbox{st}(xV, \mathcal{B}(V))\subset xU$ for any $x\in G$. Indeed, take any fixed $\in\mathcal{B}(e)$; since $G$ is a strongly pre-topological group, there exists $V\in\mathcal{B}(e)$ such that $VV^{-1}V\subset U$. Then it easily check that $\mbox{st}(xV, \mathcal{B}(V))\subset xU$ for any $x\in G$.
\end{proof}

However, the following question is still unknown for us.

\begin{question}
Is each pre-topological group $G$ completely regular?
\end{question}

\section{The coreflection of pre-uniform spaces}
If $\mu$ is a pre-uniformity on a set $X$, then the family $\{\bigcap_{U\in\mathscr{F}}U: \mathscr{F}\subseteq\mu, |\mathscr{F}|<\omega\}$ is a pre-base for a pre-uniformity $\mu^{\ast}$, which is the coarsest uniformity containing $\mu$; we say that $\mu^{\ast}$ is a uniform coreflection of $\mu$. If $\tau(\mu)$ is pre-topology induced by $\mu$, then we denote $\tau(\mu^{\ast})$ as the topology induced by $\mu^{\ast}$. Clearly, the following proposition is obvious.

\begin{proposition}\label{ppp}
Let $\mu$ be a pre-uniformity on a set $X$. Then $\bigcap\mu=\bigtriangleup$ if and only if $\bigcap\mu^{\ast}=\bigtriangleup$; hence if $(X, \tau(\mu))$ is $T_{0}$ then $(X, \tau(\mu^{\ast}))$ is completely regular.
\end{proposition}

By Theorem~\ref{t7}, we have the following proposition.

\begin{proposition}
Let $\mu$ be a strongly pre-uniformity on a set $X$. Then the following statements hold:
\begin{enumerate}
\item $\mu^{\ast}$ has a pre-base $\{\bigcap_{U\in\mathscr{F}}U: \mathscr{F}\subseteq\mu, |\mathscr{F}|<\omega\}$;

\item $(X, \tau(\mu))$ is completely regular if and only if $(X, \tau(\mu^{\ast}))$ is completely regular.
\end{enumerate}
\end{proposition}

By Lemma~\ref{l3}, we have the following proposition.

\begin{proposition}
Let $(X, \mu)$ and $(Y, \nu)$ be pre-uniform spaces. If $f: (X, \mu)\rightarrow(Y, \nu)$ is pre-uniformly continuous, then $f: (X, \tau(\mu^{\ast}))\rightarrow (Y, \tau(\nu^{\ast}))$ is continuous.
\end{proposition}

\begin{proposition}
Let $(X, \mu)$ and $(Y, \nu)$ be pre-uniform spaces and assume that $(X, \tau(\mu^{\ast}))$ is a compact Hausdorff space. If $f: (X, \tau(\mu))\rightarrow (Y, \tau(\nu))$ is pre-continuous, then $f: (X, \mu)\rightarrow (Y, \nu)$ is pre-uniformly pre-continuous.
\end{proposition}

\begin{proof}
Clearly, it suffices to prove that $f_{2}^{-1}[V]\in\mu$ for each $V\in\nu$. Take any $V\in\nu$. Then there exist $W, L\in\nu$ such that $W\circ L\subseteq V$. For each $x\in X$, since $f: (X, \tau(\mu))\rightarrow (Y, \tau(\nu))$ is pre-continuous, there exist $O_{1}, O_{1}\in\tau(\mu)$ such that $x\in O_{1}\cap O_{1}$ and $f(O_{1})\subseteq W[f(x)]$ and $f(O_{2})\subseteq L^{-1}[f(x)]$. Then it is easily verified that $(x, x)\in O_{2}\times O_{1}\subseteq f_{2}^{-1}[V]$, which implies that $f_{2}^{-1}[V]$ is a neighborhood of $\bigtriangleup$ in $(X, \tau(\mu))\times (X, \tau(\mu))$. We claim that each $(X, \tau(\mu))\times (X, \tau(\mu))$ neighborhood of $\bigtriangleup$ belongs to $\mu$. Suppose not, then there exists a $(X, \tau(\mu))\times (X, \tau(\mu))$ neighborhood $V$ of $\bigtriangleup$ which is not a member of $\mu$. Put $\eta=\{U-V: U\in\mu^{\ast}\}$. Then $\eta$ is a base for a filter on $X\times X$ and $\mu^{\ast}$ is coarser than $\eta$. Since  $(X, \tau(\mu^{\ast}))$ is a compact Hausdorff space, $\eta$ has a cluster point $(x, y)$ in $(X, \tau(\mu^{\ast}))\times (X, \tau(\mu^{\ast}))$ such that $x\neq y$, hence $(x, y)$ is a cluster point of $\mu^{\ast}$. However, it follows from Lemma~\ref{l6} and Proposition~\ref{ppp} that the intersection of the closures of members of $\mu^{\ast}$ is $\bigtriangleup$, which is a contradiction.
\end{proof}

\begin{definition}
Let $\{(X_{\alpha}, \mu_{\alpha})\}_{\alpha\in I}$ be a family of pre-uniform spaces and let $X=\prod_{\alpha\in I}X_{\alpha}$. The {\it product pre-uniformity} is the coarsest pre-uniformity on $X$ for which each projection $\pi_{\alpha}: X\rightarrow X_{\alpha}$ is pre-uniformly continuous.
\end{definition}

The family of all sets of the form $\{(x, y): (\pi_{\alpha}(x), \pi_{\alpha}(y))\in U_{\alpha}, \alpha\in F\}$, where $F$ is a finite subset of $I$ and $U_{\alpha}\in\mu_{\alpha}$ for any $\alpha\in F$, is a pre-base for the product pre-uniformity. In particular, if $(X, \mu)$ and $(Y, \nu)$ are pre-uniform spaces, a pre-base for the product pre-uniformity on $X\times Y$ consists of the family of relations on $X\times Y$ to which $B$ belongs if there are $U\in\mu$ and $V\in\nu$ such that $B[(x, y)]=U[x]\times V[y]\}$ for each $(x, y)\in X\times Y$.

The following two propositions are easily checked.

\begin{proposition}
Let $f: (X=\prod_{\alpha\in I}X_{\alpha}, \mu)\rightarrow \prod_{\alpha\in I}(X_{\alpha}, \mu_{\alpha})$. Then $f$ is pre-uniformly continuous if and only if for every $\alpha\in I$, $\pi_{\alpha}\cdot f$ is pre-uniformly continuous.
\end{proposition}

\begin{proposition}
Let $\mu$ be the product pre-uniformity on the family of pre-uniform spaces $\{(X_{\alpha}, \mu_{\alpha})\}_{\alpha\in I}$. Then $\mu^{\ast}$ is just the product uniformity on the family of uniform spaces $\{(X_{\alpha}, \mu_{\alpha}^{\ast})\}_{\alpha\in I}$.
\end{proposition}

Let $(X, \mu)$ be a pre-uniform space, $U\in\mu$ and $A\subseteq X$. We say that $A$ is $U$-dense if for each $x\in X$ there exists a point $y\in A$ such that $(x, y)\in U;$ further, we say that $(X, \mu)$ is {\it totally bounded} if for each $U\in\mu$ there exists a finite subset $A\subseteq X$ that is $U$-dense in $(X, \mu)$.

The following proposition is easily checked.

\begin{proposition}
Let $(X, \mu)$ be a totally bounded pre-uniform space. If for each $U\in\mu$ there exists $W\in\mu$ such that $W\cdot W\subseteq U$, then for each $U\in\mu$ there is a finite cover $\mathcal{A}_{U}$ of $X$ such that for each $A\in\mathcal{A}_{U}$, $A\times A\subseteq U$.
\end{proposition}

\begin{corollary}\label{c202211}
Let $(X, \mu)$ be a totally bounded strongly pre-uniform space. Then for each $U\in\mu$ there is a finite cover $\mathcal{A}_{U}$ of $X$ such that for each $A\in\mathcal{A}_{U}$, $A\times A\subseteq U$.
\end{corollary}

\smallskip
\section{Pre-proximities and pre-proximity spaces}
In this section, we shall introduce the concept of pre-proximity on a set and pre-proximity spaces, and then we discuss some basic properties of pre-proximity and pre-proximity spaces. First, we give the following concept of pre-proximity on set, which is a generalization of proximity.

Let $X$ be a set and $\delta$ a relation on $\mathcal{P}(X)$. We shall write $A\delta B$ if the sets $A, B\in\mathcal{P}(X)$ are $\delta$-related, otherwise we shall write $A\overline{\delta}B$. We say that a relation $\delta$ on $\mathcal{P}(X)$ is a {\it pre-proximity} on the set $X$ if $\delta$ satisfies the following conditions (PP1)-(PP5):

\smallskip
(PP1) $A\delta B$ if and only if $B\delta A$.

\smallskip
(PP2) If $A\delta B$ and $B\subseteq C$, then $A\delta C$.

\smallskip
(PP3) $\{x\}\delta\{y\}$ if and only if $x=y$.

\smallskip
(PP4) $\emptyset\bar{\delta}X$.

\smallskip
(PP5) If $A\bar{\delta} B$, then there exists $C\in\mathcal{P}(X)$ such that $A\bar{\delta} C$ and $B\bar{\delta} (X\setminus C)$.

{\it A pre-proximity space} is a pair $(X, \delta)$ which consists of a set $X$ and a pre-proximity $\delta$ on the set $X$. If $A\delta B$, then $A$ is said to be {\it near} $B$ and if $A\bar{\delta} B$, then $A$ is said to be {\it far from} $B$. A pre-proximity space is proximity if the following condition (PP6) holds:

\smallskip
(PP6) $A\delta B\cup C$ if and only if $A\delta B$ or $A\delta C$.

\smallskip
From the definition of pre-proximity, we have the following proposition.

\begin{proposition}\label{pp10}
Let $\delta$ be pre-proximity on the set $X$. Then we have the following two statements.

\smallskip
(1) If $A\cap B\neq\emptyset$, then $A\delta B$.

\smallskip
(2) For each $A\in\mathcal{P}(X)$, we have $\emptyset\bar{\delta}A$.

\smallskip
(3) If $A\subseteq A^{\prime}$, $B\subseteq B^{\prime}$ and $A\delta B$, then $A^{\prime}\delta B^{\prime}$.
\end{proposition}

\begin{proof}
To establish (1), let $A\cap B\neq\emptyset$, and take any $x\in A\cap B$, then $\{x\}\delta\{x\}$, hence $\{x\}\delta A$ and $A\delta\{x\}$ by (PP2) and (PP1) respectively, and $A\delta B$ by (PP2) again. Property (2) follows from (PP4) and (PP2). Property (3) follows from (PP1) and (PP2).
\end{proof}

For each pre-proximity $\delta$ on the set $X$, we can induce a pre-topology $\mathcal{P}$ on the set $X$. Indeed, for each $A\in\mathcal{P}(X)$, put $$\bar{A}=\{x\in X: x\delta A\},$$which defines a closure operator on the set $X$ satisfying the conditions (a)-(d) in \cite[Theorem 8]{LCL}. In order to prove it, we need the following lemma.

\begin{lemma}\label{lll}
For each pre-proximity $\delta$ on the set $X$ and any $A, B\in\mathcal{P}(X)$, if $B\bar{\delta}A$, then $B\bar{\delta}\bar{A}$.
\end{lemma}

\begin{proof}
Let $B\bar{\delta}A$. From (PP5), there exists $C\in\mathcal{P}(X)$ such that $B\bar{\delta}C$ and $A\bar{\delta}(X\setminus C)$. We claim that $\bar{A}\subseteq C$. Indeed, take any $x\in \bar{A}$; then $\{x\}\delta A$, hence $A\delta \{x\}$ by (PP1). Assume $x\in X\setminus C$, then it follows from (PP2) that $A\delta (X\setminus C)$, which is a contradiction. Therefore, $\overline{A}\subset C$, then $B\bar{\delta}\bar{A}$ since $B\bar{\delta}C$.
\end{proof}

\begin{theorem}\label{ct}
For each pre-proximity $\delta$ on the set $X$, the closure operator $c: \mathcal{P}(X)\rightarrow \mathcal{P}(X)$, which is defined by $c(A)=\{x\in X: x\delta A\}$ for each $A\in\mathcal{P}(X)$, satisfies the following conditions:

\smallskip
(a) $c(\emptyset)=\emptyset$.

\smallskip
(b) For every $A\in\mathcal{P}(X)$, we have $A\subseteq c(A)$.

\smallskip
(c) For every $A\in\mathcal{P}(X)$, we have $c(c(A))= c(A)$.

\smallskip
(d) If $A\subseteq B$, then $c(A)\subseteq c(B)$.
\end{theorem}

\begin{proof}
Conditions (a), (b) and (d) follows from (2) of Proposition~\ref{pp10}, (1) of Proposition~\ref{pp10} and (PP2) respectively. Now we only need to prove (c). By (b), it suffices to prove $c(c(A))\subseteq c(A)$. Assume $x\notin c(A)$, then $\{x\}\bar{\delta} A$. Hence it follows from Lemma~\ref{lll} that $\{x\}\bar{\delta}\bar{A}$, that is, $\{x\}\bar{\delta}c(A)$, which implies that $x\not\in c(c(A))$. Therefore, $c(c(A))=c(A).$
\end{proof}

Then it follows from \cite[Theorem 8]{LCL} that the family $$\mathcal{P}=\{U: c(X\setminus U)=X\setminus U\}$$ generated by $c$ in Theorem~\ref{ct} is a pre-topology. From (PP4), $(X, \mathcal{P})$ is a $T_{1}$-space. The pre-topology $\mathcal{P}$ is called {\it the pre-topology induced by the pre-proximity $\delta$} (or simply the pre-topology $\mathcal{P}(\delta)$ of $\delta$). Moreover, it is easily verified that, for any $A, B\in\mathcal{P}(X)$,$$A\delta B\ \mbox{if and only if}\ \bar{A}\delta\bar{B}.$$ A pre-topological space $(X, \tau)$ is said to be admit a pre-proximity $\delta$ provided $\delta$ induces $\tau$, and $\delta$ is said to be {\it compatible} with $\tau$.

\begin{lemma}\label{l21}
Let $(X, \delta)$ be a pre-proximity space. If $\{x\}\bar{\delta}A$, then there exists a $\mathcal{P}(\delta)$-neighborhood $U$ of $x$ such that $U\bar{\delta} A.$
\end{lemma}

\begin{proof}
Since $\{x\}\bar{\delta}A$, it follows from (PP5) that there exists $C\subseteq X$ such that $\{x\}\bar{\delta}C$ and $A\bar{\delta}(X\setminus C)$. Put $U=(X\setminus C)^{\circ}$. Then $U\bar{\delta} A$, and $x\in U$ since $x\not\in \mbox{cl}_{\delta}(C)$ by $\{x\}\bar{\delta}C$ and \cite[Theorem 9]{LCL}.
\end{proof}

\begin{theorem}\label{t718}
For each pre-uniformity $\mu$ on the set $X$ and any $A, B\in\mathcal{P}(X)$, we define $A\delta_{\mu} B$ whenever $V\cap (A\times B)\neq\emptyset$ for any $V\in\mu$. Then $\delta_{\mu}$ is a pre-proximity on the set $X$. The pre-topology induced by $\delta_{\mu}$ coincides with the pre-topology induced by $\mu$.
\end{theorem}

\begin{proof}
From the definition, it easily see that $A\delta_{\mu} B$ if and only if for any $V\in\mu$ there exist $x\in A$ and $y\in B$ such that $(x, y)\in V$. Then it is obvious that (PP1), (PP2) and (PP4) hold. From (U5), it follows that (PP3) holds. Finally, assume that $A\bar{\delta}_{\mu}B$, then there exist $V, W_{1}, W_{2}\in\mu$ such that $V\cap (A\times B)=\emptyset$ and $W_{1}\cdot W_{2}^{-1}\subseteq V$. Put $$C=X\setminus (\bigcup_{x\in A}W_{1}[x])\mbox{and}\ D=X\setminus (\bigcup_{x\in B}W_{2}[x]).$$ Then $W_{1}\cap (A\times C)=\emptyset$,  $W_{2}\cap (B\times D)=\emptyset$ and $C\cup D=X$ since $$(\bigcup_{x\in A}W_{1}[x])\cap(\bigcup_{x\in B}W_{2}[x])=\emptyset,$$ hence $A\bar{\delta_{\mu}}C$ and $B\bar{\delta_{\mu}}D$, thus $A\bar{\delta_{\mu}}C$ and $B\bar{\delta_{\mu}}(X\setminus C)$ since $D\subseteq X\setminus C$.
\end{proof}

The pre-proximity $\delta_{\mu}$ in Theorem~\ref{t718} is called the {\it pre-proximity} induced by the pre-uniformity $\mu$. A pre-uniformity $\mu$ is said to be {\it compatible} with $\delta$ if $\delta_{\mu}=\delta$. If $\delta$ is a pre-proximity on $X$, then $\pi(\delta)$ denotes the class of all pre-uniformities compatible with $\delta$. Two pre-uniformities that belong to the same pre-proximity class are called {\it pp-equivalent}. The following proposition is obvious.

\begin{proposition}\label{p20222024}
Let $\mu$ and $\nu$ be pre-uniformities on a set $X$. If $\mu\subseteq\nu$, then $\tau(\mu)\subseteq\tau(\nu)$ and $\delta_{\nu}\subseteq\delta_{\mu}$.
\end{proposition}

\begin{proposition}
Let $\delta$ be a pre-proximity on a set $X$ and $A$ be a subset of $X$. Then $\delta_{E}=\delta\cap (\mathcal{P}(A)\times\mathcal{P}(A))$ is a pre-proximity on $E$ and $\mathcal{P}(\delta_{E})=\mathcal{P}(\delta)|{E}$. Further, if $\mu$ induces $\delta$, then $\mu|A\times A$ induces $\delta_{E}$.
\end{proposition}

\begin{definition}
A set $B$ is called a $\delta$-neighborhood of a set $A$ provided $A\bar{\delta}(X\setminus B)$.
\end{definition}

\begin{proposition}\label{p20222026}
Let $\delta$ be a pre-proximity on a set $X$, and let $\ll$ be a the relation on $\mathcal{P}(X)$ defined by $A\ll B$ iff $B$ is a $\delta$-neighborhood of $A$. Then $\ll$ has the following properties (in (PSI5) and (PSI6) the pre-topology induced by $\delta$ is being considered):

\smallskip
(PSI1) If $A\ll B$, then $X\setminus B\ll X\setminus A.$

\smallskip
(PSI2) If $A\ll B$, then $A\subseteq B$.

\smallskip
(PSI3) If $A\subseteq B\ll C\subseteq D$, then $A\ll D$.

\smallskip
(PSI4) $\emptyset\ll \emptyset$ and $X\ll X$

\smallskip
(PSI5) If $A\ll B$, then there exists open set $U\subseteq X$ such that $A\ll U\subseteq \overline{U}\ll B$.

\smallskip
(PSI6) For each $x\in X$ and any neighborhood $A$ of $x$ we have $\{x\}\ll A$.

Conversely, let a relation $\ll$ satisfying conditions (PSI1)-(PSI5) be defined on $\mathcal{P}(X)$. The relation $\delta$ defined by $A\bar{\delta}B$ if $A\ll X\setminus B$ is a pre-proximity on $X$. Further, $B$ is a $\delta$-neighborhood of $A$ iff $A\ll B$.
\end{proposition}

\begin{proof}
Clearly, (PSI1)-(PSI4) are obvious. It suffice to prove (PSI5) and (PSI6). First, it follow from (PP5) that we have the following claim:

\smallskip
{\bf Claim:} If $A\bar{\delta} B$, then there exist $C, D\subseteq X$ such that $A\ll C, B\ll D$ and $C\cap D=\emptyset$.

\smallskip
(PSI5). Indeed, we can prove a more stronger result. By Claim, there exists $C, D\subseteq X$ such that $A\ll C, X\setminus B\ll D$ and $C\cap D=\emptyset$. By (PSI1), we have $C\subseteq X\setminus D\ll B$, thus $C\bar{\delta}(X\setminus B)$. Since $A\bar{\delta}(X\setminus C)$, we conclude that $A\bar{\delta}\overline{(X\setminus C)}$ from Lemma~\ref{lll}. Hence $A\bar{\delta}(X\setminus C^{\circ})$ by \cite[Theorem 9]{LCL}. Put $U=C^{\circ}$. Then we have $A\ll U\subseteq C$; since $C\bar{\delta}(X\setminus B)$, it follows that $U\bar{\delta}(X\setminus B)$, which shows that $\bar{U}\bar{\delta}(X\setminus\{B\})$, thus $\overline{U}\ll B$.

(PSI6). Assume that $\{x\}\ll A$ does not hold, then $\{x\}\delta X\setminus A$, thus $x\in \overline{X\setminus A}$. Therefore, $A\cap (X\setminus A)\neq\emptyset$ since $A$ is a neighborhood of $x$, which is a contradiction.

\smallskip
Conversely, it is easily checked that $\delta$ is a pre-proximity.
\end{proof}

Let $X$ be a set, and let $A, B\in\mathcal{P}(X)$. Denote the set $X\times X-A\times B$ by $T(A, B)$.

\begin{proposition}\label{p20222023}
Let $(X, \delta)$ be a pre-proximity space,  and let $$\mathcal{S}=\{T(A, B): A\bar{\delta}B, A, B\in\mathcal{P}(X)\}.$$ Then $\mathcal{S}$ is a pre-base for a totally bounded pre-uniformity $\mu_{\delta}$ which is compatible with $\delta$. Moreover, $\mu_{\delta}$ is the coarsest pre-uniformity in $\pi(\delta)$.
\end{proposition}

\begin{proof}
Clearly, $\mathcal{S}$ is a pre-base for a pre-uniformity $\mu_{\delta}$. Indeed, it is easily checked that $\mathcal{S}$ satisfies the conditions (U1) and (U2). We only need to prove that $\mathcal{S}$ satisfies the conditions (U3) and (U5). Take any $T(A, B)\in\mathcal{S}$. Then $A\bar{\delta}B$, hence there exists a subset $C$ of $X$ such that $A\bar{\delta}C$ and $X-C\bar{\delta}B$. We claim that $T(A, C)\circ T(X-C, B)\subseteq T(A, B).$ In fact, take any $(x, z)\in T(A, C)$ and $(z, y)\in T(X-C, B)$. Suppose that $x\in A$. Then $z\not\in C$, hence $z\in X-C$, thus $y\not\in B$, which implies that $(x, y)\in T(A, B).$ Therefore, $T(A, C)\circ T(X-C, B)\subseteq T(A, B).$ Finally, from (PP3) it is easily checked that $\bigcap\mathcal{S}=\triangle$.

Now we conclude that $\mu_{\delta}$ is totally bounded. Then it suffices prove that each element $U$ of $\mathcal{S}$ is $U$-dense. Take any $U=T(A, B)\in\mathcal{S}$. Then $A\bar{\delta}B$, hence $A\cap B=\emptyset$. Take any point $a\in A$ and any point $b\in B$, and put $C=\{a, b\}$. For any $x\in X$, without loss of generality, we may assume that $x\not\in B$, then $$(x, a)\in (X-B)\times (X-B)\subseteq T(A, B)$$ since $$(X-A)\times (X-A)\cup (X-B)\times (X-B)\subseteq T(A, B).$$Therefore, $T(A, B)$ is $U$-dense.

Now we prove that $\mu_{\delta}$ is compatible with $\delta$. Let $\theta$ be the pre-proximity induced by $\mu_{\delta}$. For any $A, B\in\mathcal{P}(X)$, we conclude that $A\delta B$ if and only if $A\theta B$. Indeed, if $A\bar{\delta}B$, then $A\times B\cap T(A, B)=\emptyset$, hence $A\bar{\theta}B$. Now let $A\delta B$. We assume that $A\bar{\theta}B$, then there exist $E, F\in\mathcal{P}(X)$ such that $A\times B\cap T(E, F)=\emptyset$ and $E\bar{\delta}F$. Thus $A\times B\subseteq E\times F,$ then $A\subseteq E$ and $B\subseteq F$. However, since $A\delta B$, it follows form definition of pre-proximity that $E\delta F$, which is a contradiction. Therefore, $A\theta B$.

$\mu_{\delta}$ is the coarsest pre-uniformity compatible with $\delta$. Let $\mu\in \pi(\delta)$. Let $A\bar{\delta}B$. Then there exists $U\in\mu$ such that $A\times B\cap U=\emptyset$, thus $U\subseteq T(A, B)$. Therefore, we have $\mu_{\delta}\subseteq \mu$.
\end{proof}

If $\mu$ is a pre-uniformity on a set $X$, then $\mu_{w}$ denotes the totally bounded member of $\pi(\delta_{\mu})$. From Proposition~\ref{p20222023}, it follows that $\mu_{w}$ is the finest totally bounded pre-uniformity on $X$ that is coarser than $\mu$.

However, the following questions are still unknown for us.

\begin{question}
Is $\mu_{\delta}$ the unique totally bounded pre-uniformity compatible with $\delta$  in Proposition~\ref{p20222023}?
\end{question}

\begin{question}
If two pre-uniformities $\mu$ and $\nu$ pp-equivalent, does $\mu_{w}=\nu_{w}$ hold?
\end{question}

By Proposition~\ref{p20222024}, we consider the family of all pre-proximities on a set $X$ to be partially order by {\it reverse set} inclusion and we say that $\rho$ is {\it finer} than $\delta$ (and $\delta$ is {\it coarser} than $\rho$) provided $\rho\subset \delta$. Each set $X$ has a finest pre-proximity: the discrete pre-proximity given by $A\delta B$ iff $A\cap B\neq\emptyset$. Of course, it has a coarsest pre-proximity given by $A\delta B$ iff $A\neq\emptyset$ and $B\neq\emptyset$.

The following proposition is obvious.

\begin{proposition}
Let $\{\delta_{i}: i\in I\}$ be a nonempty family of pre-proximities on a set $X$ and let $\delta_{0}$ be defined by $A\delta_{0} B$ iff for every finite cover $\mathcal{A}$ and every finite cover $\mathcal{B}$ of $B$ there are $A^{\prime}\in\mathcal{A}$ and $B^{\prime}\in\mathcal{B}$ such that $A^{\prime}\delta_{i} B^{\prime}$ for each $i\in I$. Then $\delta_{0}$ is the coarsest pre-proximity on $X$ such that it is finer than $\delta_{i}$ for each $i\in I$.
\end{proposition}

Suppose that $\{\mu_{i}: i\in I\}$ is a collection of pre-uniformities on a set $X$ and let $\mu=\sup\{\mu_{i}: i\in I\}$. It is natural to ask the following question.

\begin{question}\label{p20222025}
Does $\sup\{\delta_{\mu_{i}}: i\in I\}=\delta_{\mu}$ hold?
\end{question}

The following proposition gives a partial answer to Question~\ref{p20222025}.

\begin{proposition}
Let $\{\mu_{i}: i\in I\}$ be a collection of totally bounded pre-uniformities on a set $X$, let $\mu=\sup\{\mu_{i}: i\in I\}$ and let $\delta_{0}=\sup\{\delta_{\mu_{i}}: i\in I\}$. Then $\delta_{0}=\delta_{\mu}.$
\end{proposition}

\begin{proof}
For every $i\in I$, we have $\mu_{i}\subseteq \mu$ and $\delta_{\mu}\subseteq \delta_{\mu_{i}}$. Therefore, $\delta_{\mu}\subseteq \delta_{0}$. Note that if $A\bar{\delta}_{\mu_{i}} B$, then $A\bar{\delta_{0}}B$ and that if $T(A, B)\in\mu_{i}$, then $T(A, B)\in\mu_{\delta_{0}}$. From Proposition~\ref{p20222023}, it follows that $\mu_{i}\subseteq \mu_{\delta_{0}}$, which shows that $\mu\subseteq \mu_{\delta_{0}}$. Thus $\delta_{0}\subseteq \delta_{\mu}$.
\end{proof}

\begin{corollary}
Suppose that $\{\delta_{i}: i\in I\}$ is a family of pre-proximities on a set $X$, then $\sup\{\delta_{i}: i\in I\}$ induces $\sup\{\mathcal{P}(\delta_{i}): i\in I\}$.
\end{corollary}

\begin{proposition}
Let $(X, \tau)$ be a normal Hausdorff pre-topological space. The relation defined by $A\delta B$ iff $\overline{A}\cap \overline{B}\neq\emptyset$ is the finest pre-proximity compatible with $\tau$.
\end{proposition}

\begin{proof}
Since $(X, \tau)$ is a normal Hausdorff pre-topological space, it is easily verified that $\delta$ is a pre-proximity on $X$. Now we only need to prove that $\delta$ is the finest pre-proximity compatible with $\tau$.

\smallskip
(1) $\delta$ is compatible with $\tau$.

\smallskip
By Theorem~\ref{ct}, it suffices to prove that $\overline{A}=\{x\in X: \{x\}\delta A\}$ for each subset $A$ of $X$. Indeed, take any $x\in \overline{A}$. Since $(X, \tau)$ is a $T_{1}$ pre-topological space, it follows that $\{x\}$ is closed, hence $\overline{\{x\}}\cap \overline{A}=\{x\}\cap \overline{A}=\{x\}\neq\emptyset$, which shows that $\{x\}\delta A.$ Conversely, suppose that $\{x\}\delta A$. Then $\{x\}\delta \overline{A}$, which implies that $\{x\}\cap \overline{A}\neq\emptyset$, thus $x\in \overline{A}$.

\smallskip
(2) $\delta$ is the finest pre-proximity compatible with $\tau$.

\smallskip
Suppose that $\rho$, which is compatible with $\tau$, is finer than $\delta$, then $\rho\subseteq \delta$. Let $A\delta B$. Then $\overline{A}\cap \overline{B}\neq\emptyset$. Take any $x\in\overline{A}\cap \overline{B}$; then $x\in \overline{A}=\{y\in X: \{y\}\rho A\}$. Hence $\{x\}\rho A$, so $A\rho \overline{B}$, then $A\rho B$ by Lemma~\ref{lll}. Therefore, we have $\rho= \delta$.
\end{proof}

Let $\delta$ be a pre-proximity on a set $X$. A finite cover $\{A_{i}\}_{i=1}^{k}$ of the set $X$ is called {\it $\delta$-pre-uniform} if there exists a cover $\{B_{i}\}_{i=1}^{k}$ of the set $X$ such that $B_{i}\ll A_{i}$ for each $i=1, 2, \ldots, k$.

\begin{proposition}\label{p20222027}
Let $\delta$ be a pre-proximity on a set $X$ and $A, B\subseteq X$. If  each $\delta$-pre-uniform cover $\{A_{i}\}_{i=1}^{k}$ of the set $X$ contains a set $A_{j}$ such that $A\cap A_{j}\neq\emptyset\neq B\cap A_{j}$, then $A\delta B$.
\end{proposition}

\begin{proof}
Assume that $A\bar{\delta} B$. From (PP5), it follows that there exists $C\subseteq X$ such that $A\bar{\delta} C$ and $B\bar{\delta}(X\setminus C)$, then $C\ll X-A$ and $X-C\ll X-B$. Let $A_{1}=X-A, A_{2}=X-B$, $B_{1}=C$ and $B_{2}=X-C$. Since $A_{1}\cup A_{2}=X-A\cap B=X$ and $B_{1}\cup B_{2}=C\cup (X-C)=X$, $\mathcal{A}=\{A_{i}: i=1, 2\}$ is a $\delta$-pre-uniform cover of $X$. However, no member of $\mathcal{A}$ meets both $A$ and $B$, which is a contradiction.
\end{proof}

We don't know if the condition in Proposition~\ref{p20222027} is necessary.

{\bf Acknowledgements}. The authors wish to thank
the reviewers for careful reading preliminary version of this paper and providing many valuable suggestions.

\end{document}